\documentclass[reqno]{amsart}
\usepackage{amsthm, amssymb, amsfonts}
\usepackage{lmodern}
\usepackage{color}
\usepackage{hyperref}
\usepackage[T5]{fontenc}
\usepackage{amscd,amssymb}
\usepackage[v2,cmtip]{xy}
\usepackage{mathrsfs}
\theoremstyle{plain}
\newtheorem{thm}{Theorem}[section]

\theoremstyle{definition}
\newtheorem{defn}[thm]{Definition}

\theoremstyle{plain}
\newtheorem{thms}{Theorem}[subsection]
\newtheorem{props}[thms]{Proposition}
\newtheorem{lems}[thms]{Lemma}

\theoremstyle{definition}

\newtheorem{rems}[thms]{Remark}

\numberwithin{equation}{section}

\begin{document}

\title[On the determination of the Singer transfer]
{On the determination of the Singer transfer}

 \author{Nguy\~\ecircumflex n Sum} 

\address{Department of Mathematics, Quy Nh\ohorn n University, 170 An D\uhorn \ohorn ng V\uhorn \ohorn ng, Quy Nh\ohorn n, B\`inh \DJ \d inh, Viet Nam}

\email{nguyensum@qnu.edu.vn }

\thanks{This research is funded by Vietnam National Foundation for Science and Technology Development (NAFOSTED) under grant number 101.04-2017.05.}

\subjclass[2010]{Primary 55T15; Secondary 55S10, 55S05}

%\date{March 7, 2017.}

\keywords{Steenrod algebra, algebraic transfer, polynomial algebra}

\begin{abstract}
Let $P_k$ be the graded polynomial algebra $\mathbb F_2[x_1,x_2,\ldots ,x_k]$ with the degree of each generator $x_i$
being 1, where $\mathbb F_2$ denote the prime field of two elements, and let $GL_k$ be the general linear group over $\mathbb F_2$ which acts regularly on $P_k$.

We study the algebraic transfer constructed by Singer \cite{si1} using the technique of the \textit{Peterson hit problem}. This transfer is a homomorphism from the homology of the mod-2 Steenrod algebra $\mathcal A$, $\text{Tor}^{\mathcal A}_{k,k+d} (\mathbb F_2,\mathbb F_2)$, to the subspace of $\mathbb F_2{\otimes}_{\mathcal A}P_k$ consisting of all the $GL_k$-invariant classes of degree $d$. 

In this paper, by using the results on the Peterson hit problem we present the proof of the fact that the Singer algebraic transfer is an isomorphism for $k \leqslant 3$. This result has been proved by Singer in \cite{si1} for $k \leqslant 2$ and by Boardman in \cite{bo} for $k = 3$. We show that the fourth Singer transfer is also an isomorphism in certain internal degrees. This result is new and it is different from the ones of Bruner, H\`a and H\uhorn ng~\cite{br}, Ch\ohorn n and H\`a~\cite{cha2}, H\`a ~\cite{ha}, H\uhorn ng and Qu\`ynh~ \cite{hq}, Nam~ \cite{na2}. 
\end{abstract}

\maketitle

%======================================
\section{Introduction}\label{s1} 
\setcounter{equation}{0}

 Denote by $P_k:= \mathbb F_2[x_1,x_2,\ldots ,x_k]$ the polynomial algebra over the field of two elements, $\mathbb F_2$, in $k$ generators $x_1, x_2, \ldots , x_k$, each of degree 1. This algebra arises as the cohomology with coefficients in $\mathbb F_2$ of an elementary abelian 2-group of rank $k$.  Therefore, $P_k$ is a module over the mod-2 Steenrod algebra, $\mathcal A$.
The action of $\mathcal A$ on $P_n$ is determined by the elementary properties of the Steenrod squares $Sq^i$ and subject to the Cartan formula
$Sq^k(fg) = \sum_{i=0}^kSq^i(f)Sq^{k-i}(g),$
for $f, g \in P_k$ (see Steenrod and Epstein~\cite{st}).

The \textit{Peterson hit problem} is to find a minimal generating set for $P_k$ regarded as a module over the  mod-2 Steenrod algebra. Equivalently, this problem is to find a vector space basis for $QP_k := \mathbb F_2 \otimes_{\mathcal A} P_k$ in each degree $d$. Such a basis may be represented by a list of monomials of degree $d$.  It is completely determined for $k \leqslant 4$, unknown in general. 

Let $GL_k$ be the general linear group over the field $\mathbb F_2$. This group acts naturally on $P_k$ by matrix substitution. Since the two actions of $\mathcal A$ and $GL_k$ upon $P_k$ commute with each other, there is an inherited action of $GL_k$ on $QP_k$. 

Denote by $(P_k)_d$ the subspace of $P_k$ consisting of all the homogeneous polynomials of degree $d$ in $P_k$ and by $(QP_k)_d$ the subspace of $QP_k$ consisting of all the classes represented by the elements in $(P_k)_d$. 
In \cite{si1}, Singer defined the algebraic transfer,  which is a homomorphism
$$\varphi_k :\text{Tor}^{\mathcal A}_{k,k+d} (\mathbb F_2,\mathbb F_2) \longrightarrow  (QP_k)_d^{GL_k}$$
from the homology of the Steenrod algebra to the subspace of $(QP_k)_d$ consisting of all the $GL_k$-invariant classes. It is a useful tool in describing the homology groups of the Steenrod algebra, $\text{Tor}^{\mathcal A}_{k,k+d} (\mathbb F_2,\mathbb F_2)$. This transfer was studied by  Boardman~\cite{bo}, Bruner, H\`a and H\uhorn ng~\cite{br}, H\`a ~\cite{ha}, H\uhorn ng ~\cite{hu}, Ch\ohorn n and H\`a~ \cite{cha,cha1,cha2}, Minami ~\cite{mi}, Nam~ \cite{na2}, H\uhorn ng and Qu\`ynh~ \cite{hq}, the present author \cite{su4} and others.

Singer showed in \cite{si1} that $\varphi_k$ is an isomorphism for $k=1,2$. Boardman showed in \cite{bo} that $\varphi_3$ is also an isomorphism.  However, for any $k \geqslant 4$,  $\varphi_k$ is not a monomorphism in infinitely many degrees (see Singer \cite{si1}, Bruner, H\`a and H\uhorn ng~\cite{br}, H\uhorn ng \cite{hu}.) Singer made a conjecture in \cite{si1} that \textit{the algebraic transfer $\varphi_k$ is an epimorphism for any $k \geqslant 0$.} This conjecture is true for $k\leqslant 3$. It can be verified for $k=4$ by using the results in \cite{su2,su5}. The conjecture  for $k\geqslant 5$ is an open problem.

In this paper, by using the results on the Peterson hit problem we present the proof of the fact that the Singer algebraic transfer is an isomorphism for $k \leqslant 3$. Recall that this result has been proved by Singer in \cite{si1} for $k \leqslant 2$ and by Boardman in \cite{bo} for $k = 3$. To prove this result, Boardman \cite{bo} computed the space $QP_3^{GL_3}$ by using a basis consisting of the all the classes represented by certain polynomials in $P_3$. We also compute this space, however we use the admissible monomial basis for $QP_3$ that is different from the one of Boardman in \cite{bo}.  
By applying this technique for $k=4$, we show that the fourth Singer transfer is also an isomorphism in certain internal degrees. This result is new and it is different from the ones of Bruner, H\`a and H\uhorn ng~\cite{br}, Ch\ohorn n and H\`a~\cite{cha2},  H\`a ~\cite{ha}, H\uhorn ng and Qu\`ynh~ \cite{hq}, Nam~ \cite{na2}. In those works it is shown only that the fourth Singer transfer detects certain families of elements in $\text{Ext}_{\mathcal A}^{4,*}(\mathbb F_2,\mathbb F_2)$, and fails to detect others.

\medskip
This paper is organized as follows. In Section \ref{s2}, we recall  some needed information on the lambda algebra and the Singer algebraic transfer.  In Sections \ref{s3}, we present the determination of the algebraic transfer for $k \leqslant 3$. Finally, in Section \ref{s4}, we show that the fourth Singer transfer is an isomorphism in certain internal degrees. 

%====================================
%\section{Preliminaries}\label{s2}
\section{The Singer algebraic transfer and the lambda algebra}\label{s2}
\setcounter{equation}{0}

First of all, we briefly recall the definition of the Singer transfer. Let $\widehat P_1 $ be the submodule of $\mathbb F_2[x_1,x_1^{-1} ]$ spanned by all powers $x^i_1$ with $i \geqslant -1.$ The usual $\mathcal A$-action on $P_1 = \mathbb F_2[x_1]$ is canonically extended to an $\mathcal A$-action on $\mathbb F_2[x_1,x_1^{-1} ]$ (see Singer \cite{si1}).  $\widehat P_1 $ is an $\mathcal A$-submodule of $\mathbb F_2[x_1,x_1^{-1} ]$. The inclusion $P_1 \subset  \widehat P_1$ gives rise to a short exact sequence of $\mathcal A$-modules:
$$0 \longrightarrow P_1 \longrightarrow \widehat P_1  \longrightarrow \Sigma^{-1}\mathbb F_2 \longrightarrow 0 .$$
Let $e_1$ be the corresponding element in $\text{Ext}_{\mathcal A}^1(\Sigma^{-1}\mathbb F_2,P_1)$. By using the cross and Yoneda products, Singer set 
$$e_k = (e_1\times P_{k-1})\circ(e_1\times P_{k-2})\circ  \ldots (e_1\times P_{1})\circ e_1 \in \text{Ext}^k_{\mathcal A}(\Sigma^{-k}\mathbb F_2,P_k).$$ 
Then, he defined 
$\hat{\varphi_k} :\text{Tor}^{\mathcal A}_k (\mathbb F_2,\Sigma^{-k}\mathbb F_2) \longrightarrow
\text{Tor}^{\mathcal A}_0 (\mathbb F_2,P_k)=QP_k$ 
by $\hat\varphi_k(z) = e_k\cap z$. Its image is a submodule of $(QP_k)^{GL_k}$. So, $\hat\varphi_k$ induces the homomorphism
$$\varphi_k :\text{Tor}^{\mathcal A}_k (\mathbb F_2,\Sigma^{-k}\mathbb F_2) \longrightarrow
QP_k^{GL_k}.$$ 

Denote by $(P_k)^*$ the dual of $P_k$ and by $P((P_k)^*)$ the primitive subspace consisting of all elements in $(P_k)^*$ that are annihilated by every positive degree operations in the mod-2 Steenrod algebra. The dual of $\varphi_k$:
$$Tr_k := (\varphi_k)^*: \mathbb F_2\otimes_{GL_k}P((P_k)^*_d) \longrightarrow \text{\rm Ext}_{\mathcal A}^{k, k+d}(\mathbb F_2, \mathbb F_2)$$
is also called the $k$-th Singer transfer.

The algebra $\text{Ext}_{\mathcal A}^{*,*}(\mathbb F_2,\mathbb F_2)$ is described in terms of the mod-2 lambda algebra $\Lambda$ (see \cite{b6}). Recall that $\Lambda$ is a bigraded differential algebra over $\mathbb F_2$ generated by $\lambda_j \in \Lambda^{1,j}, j \geqslant 0$, with the relations
\begin{equation}\label{cth1}\lambda_j\lambda_{2j+1+m} = \sum_{\nu \geqslant 0}\binom{m-\nu-1}\nu\lambda_{j+m-\nu}\lambda_{2j+1+\nu},
\end{equation}
for $m\geqslant 0$ and the differential
\begin{equation}\label{cth2} \delta(\lambda_i) = \sum_{\nu \geqslant 0}\binom{i-\nu-1}{\nu +1}\lambda_{k-\nu-1}\lambda_{\nu},
\end{equation}
for $i>0$, $\delta(\lambda_0) =0$ and that $H^{k,d}(\Lambda,\delta) = \text{Ext}_{\mathcal A}^{k, k+d}(\mathbb F_2, \mathbb F_2)$.  

For example, the elements $\lambda_{2^i-1}\in \Lambda^{1,2^i-1}, \ i\geqslant 0,$ and $\bar d_0 = \lambda_6\lambda_2\lambda_3^2 + \lambda_4^2\lambda_3^2 + \lambda_2\lambda_4\lambda_5\lambda_3 + \lambda_1\lambda_5\lambda_1\lambda_7\in \Lambda^{4,14}$ are the cycles in the lambda algebra $\Lambda$. So, $h_i =[\lambda_{2^i-1}]$ and $d_0 = [\bar d_0]$ are the elements in $\text{Ext}_{\mathcal A}^{*, *}(\mathbb F_2, \mathbb F_2)$. Note that $h_i$ is  the Adams element in $\text{\rm Ext}_{\mathcal A}^{1,2^i}(\mathbb F_2, \mathbb F_2)$.

There is a homomorphism $\widetilde{Sq}^0 : \Lambda \to \Lambda$ determined by
$$\widetilde{Sq}^0(\lambda_{j_1}\lambda_{j_2}\ldots \lambda_{j_k}) = \lambda_{2j_1+1}\lambda_{2j_2+1}\ldots \lambda_{2j_k+1}, \ k \geqslant 0.$$
This homomorphism respects the relations in \eqref{cth1} and commutes the differential in \eqref{cth2}. Therefore, it induces a homomorphism 
$$Sq^0 :\text{Ext}_{\mathcal A}^{k, k+d}(\mathbb F_2, \mathbb F_2) = H^{k,d}(\Lambda) \longrightarrow H^{k,k+2d}(\Lambda)  = \text{\rm Ext}_{\mathcal A}^{k, 2k+2d}(\mathbb F_2, \mathbb F_2).$$

A family $\{a_i : i \geqslant 0\}$ of elements in $\text{Ext}_{\mathcal A}^{k, k+*}(\mathbb F_2, \mathbb F_2)$
is called a $Sq^0$-family if $a_i =(Sq^0)^i(a_0)$ for every $ i \geqslant 0$. 
It is well known that $\text{Ext}_{\mathcal A}^{3, 3+*}(\mathbb F_2, \mathbb F_2)$ contains the $Sq^0$-family of indecomposable elements $\{c_i\}$ and $\text{Ext}_{\mathcal A}^{4, 4+*}(\mathbb F_2, \mathbb F_2)$  contains seven $Sq^0$-families of indecomposable elements, namely $\{d_i \}, \{e_i\} , \{f_i\} , \{g_{i+1}\} , \{p_i\} , \{D_3(i)\}$, and $\{p'_i\}$.  Note that $\{h_i\}$ is also a $Sq^0$-family in $\text{\rm Ext}_{\mathcal A}^{1,1+*}(\mathbb F_2, \mathbb F_2)$. 

The algebra $\{\text{\rm Ext}_{\mathcal A}^{k, k+*}(\mathbb F_2, \mathbb F_2)|k\geqslant 0\}$ has been explicitly computed by Adem \cite{ade} for $k=1$, by Adams \cite{ada} and Wall \cite{wal} for $k=2$, by Adams \cite{ada} and Wang \cite{wan} for $k=3$ and by Lin \cite{wl} for $k = 4$.

\begin{thm}[See  \cite{ada,ade,wl,wal,wan}] \label{md51}\

{\rm i)} The algebra $\{\text{\rm Ext}_{\mathcal A}^{k, k+*}(\mathbb F_2, \mathbb F_2)| k\geqslant 0\}$ for $k\leqslant 3$ is generated by $h_i$ and $c_i$ for $i \geqslant 0$ and subject only to the relations $h_ih_{i+1} = 0, h_ih_{i+2}^2 = 0$ and $h_i^3= h_{i-1}^2h_{i+1}$. In particular, $\{c_i : i \geqslant 0\}$ is an $\mathbb F_2$-basis for the indecomposable elements in $\text{\rm Ext}_{\mathcal A}^{3, 3+*}(\mathbb F_2, \mathbb F_2)$.

{\rm ii)} The algebra $\{\text{\rm Ext}_{\mathcal A}^{k, k+*}(\mathbb F_2, \mathbb F_2)|k\geqslant 0\}$ for $k\leqslant 4$ is generated by $h_i$, $c_i$, $d_i$, $e_i$, $f_i$, $g_{i+1}$, $p_i$, $D_3(i)$ and $p'_i$ for $i \geqslant 0$ and subject to the relations
in i) together with the relations $h_i^2h_{i+3}^2 = 0$, $h_j c_i = 0$ for $j = i -1, i, i + 2$ and $i +3$. Furthermore, the set of the elements  $d_i , e_i , f_i , g_{i+1}, p_i , D_3(i)$ and $p'_i$, for $i \geqslant 0$, is an $\mathbb F_2$-basis for the indecomposable elements in $\text{\rm Ext}_{\mathcal A}^{4, 4+*}(\mathbb F_2, \mathbb F_2)$.
\end{thm}

It is well known that the dual of $P_k$ is the divided power algebra generated by $a_1,a_2,\ldots ,a_k$:
$$ (P_k)^* = \Gamma(a_1,a_2,\ldots , a_k)$$ 
where $a_j^{(i)}$ is dual to $x_j^i \in P_k$ with respect to the basis of $P_k$ consisting of all monomials in $x_1,x_2,\ldots , x_k$ and $a_j = a_j^{(1)}$. The graded vector space $\{(P_k)^*| k\geqslant 0\}$ is an algebra with a multiplication defined by 
$$(a_1^{(i_1)}\ldots a_k^{(i_k)})(a_1^{(i_{k+1})}\ldots a_m^{(i_{k+m})}) = a_1^{(i_1)}\ldots a_k^{(i_k)}a_{k+1}^{(i_{k+1})}\ldots a_{k+m}^{(i_{k+m})}\in (P_{k+m})^*,$$
for any $a_1^{(i_1)}\ldots a_k^{(i_k)}\in (P_k)^*$ and $a_1^{(i_{k+1})}\ldots a_m^{(i_{k+m})} \in (P_m)^*$.
In \cite{cha}, Ch\ohorn n and H\`a defined a homomorphism of algebras
$$\phi =\{\phi_k|k\geqslant 0\}: \{(P_k)^*| k\geqslant 0\}\longrightarrow \{\Lambda^{k,*}|k\geqslant 0\} = \Lambda,$$
which induces the Singer transfer. Here, the homomorphism $\phi_k: (P_k)^*\to\Lambda^{k,*}$ is defined by the following inductive formula:
$$\phi_k(a^{(I,t)}) = \begin{cases} \lambda_t ,&\text{ if } k-1=\ell(I) =0,\\
\sum_{i\geqslant t}\phi_{k-1}(Sq^{i-t}a^I)\lambda_i, &\text{ if } k-1= \ell(I) >0,
\end{cases}$$
for any $a^{(I,t)} = a_1^{(i_1)}a_2^{(i_2)}\ldots a_{k-1}^{(i_{k-1})}a_k^{(t)}\in (P_k)^*$ and $I=(i_1,i_2,\ldots, i_{k-1})$.
\begin{thm}[See Ch\ohorn n and H\`a \cite{cha}]\label{md52} If $b \in P((P_k)^*)$, then $\phi_k(b)$ is a cycle in the lambda algebra $\Lambda$ and $Tr_k([b]) = [\phi_k(b)]$. 
\end{thm}

Note that this theorem is a dual version of the one in H\uhorn ng \cite{hu2}.

We end this section by recalling some results on Kameko's homomorphism and the generators of the general linear group $GL_k$.

One of the main tools in the study of the hit problem is  Kameko's homomorphism 
$\widetilde{Sq}^0_*: QP_k \to QP_k$. 
This homomorphism is induced by the $\mathbb F_2$-linear map $\psi:P_k\to P_k$, given by
$$
\psi(x) = 
\begin{cases}y, &\text{if }x=x_1x_2\ldots x_ky^2,\\  
0, & \text{otherwise,} \end{cases}
$$
for any monomial $x \in P_k$. Note that $\psi$ is not an $\mathcal A$-homomorphism. However, 
$\psi Sq^{2t} = Sq^{t}\psi,$ and $\psi Sq^{2t+1} = 0$ for any non-negative integer $t$. 

For a positive integer $n$, by $\mu(n)$ one means the smallest number $r$ for which it is possible to write $n = \sum_{1\leqslant i\leqslant r}(2^{u_i}-1),$ where $u_i >0$.

\begin{thm}[Kameko~\cite{ka}]\label{dlmd2} 
Let $m$ be a positive integer. If $\mu(2m+k)=k$, then 
$(\widetilde{Sq}^0_*)_m: (QP_k)_{2m+k}\to (QP_k)_m$ 
is an isomorphism of the $GL_k$-modules.
\end{thm}

\begin{defn}\label{dfn2} Let  $f, g$ be two polynomials  of the same degree in $P_k$. Then, $f \equiv g$ if and only if $f - g \in \mathcal A^+P_k$. If $f \equiv 0$, then $f$ is called hit.
\end{defn}

For $1 \leqslant i \leqslant k$, define the $\mathcal{A}$-homomorphism  $\rho_i:P_k \to P_k$, which is determined by $\rho_i(x_i) = x_{i+1}, \rho_i(x_{i+1}) = x_i$, $\rho_i(x_j) = x_j$ for $j \ne i, i+1,\ 1 \leqslant i < k$, and $\rho_k(x_1) = x_1+x_2$,  $\rho_k(x_j) = x_j$ for $j > 1$.  

It is easy to see that the general linear group $GL_k$ is generated by the matrices associated with $\rho_i,\ 1\leqslant i \leqslant k,$ and the symmetric group $\Sigma_k$ is generated by the ones associated with $\rho_i,\ 1 \leqslant i < k$. 
So, a class $[f]$ represented by a homogeneous polynomial $f \in P_k$ is an $GL_k$-invariant if and only if $\rho_i(f) \equiv f$ for $1 \leqslant i\leqslant k$.  It is an $\Sigma_k$-invariant if and only if $\rho_i(f) \equiv f$ for $1 \leqslant i < k$. 

%========================================
\section{Determination of ${Tr_k}$ for $k\leqslant 3$}\label{s3}   

\subsection{Determination of ${Tr_k}$ for $k\leqslant 2$}\label{ss31}\   

\medskip
In this subsection, we present the proof of the following.
\begin{thms}[Singer \cite{si1}]\label{dls} The algebraic transfer $Tr_k$ is an isomorphism for $k \leqslant 2$.
\end{thms}

It is well-known that
$$(QP_1)_n^{GL_1} =  (QP_1)_n = \begin{cases} \langle [x^{2^u-1}]\rangle, &\text{if } n = 2^u-1, \ u\geqslant 0,\\
0, &\text{otherwise}. \end{cases}$$

According to Theorem \ref{md51}, we have 
$$\text{Ext}_{\mathcal A}^{1, t+1}(\mathbb F_2, \mathbb F_2)  = \begin{cases} \langle h_u\rangle, &\text{if } t = 2^u-1, \ u\geqslant 0,\\
0, &\text{otherwise}. 
\end{cases}$$
Since $(P_1)^* = \Gamma(a)$  and $a^{(2^u-1)} \in P((P_1)^*)$, $\phi_1(a^{(2^u-1)}) = \lambda_{2^u-1}$ is a cycle in $\Lambda^{1,*}$. Using Theorem \ref{md52}, we get
$$Tr_1([a^{(2^u-1)}]) = [\phi_1(a^{(2^u-1)})] = [\lambda_{2^u-1}]= h_u,\ \forall u\geqslant 0.$$
So, $Tr_1$ is a isomorphism.

Now, we present the proof of this theorem for $k=2$ by computing the space  $(QP_2)^{GL_2}$. From a result of Wood \cite{wo}, we need only to compute this space in the degree $n = 2^{s+t} + 2^s - 2$ with $s, t$ non-negative integers.

First, we consider the degree $n = 2^{s+1}-2$ with $s \geqslant 0$. Since the iterated Kameko homomorphism  $(\widetilde{Sq}^0_*)^s: (QP_2)_n \to (QP_2)_0$ is a isomorphism of  $GL_2$-modules and $(QP_2)_0^{GL_2} = \langle 1\rangle$, hence $(QP_2)_n^{GL_2} = \langle [p_{2,s}]\rangle$ with $p_{2,s} := (x_1x_2)^{2^s-1}$.

Next, we compute $(QP_2)^{GL_2}_n$ with $n = 2^{s+1} + 2^s -2$,  $s\geqslant 0$. Since the iterated Kameko homomorphism 
$(\widetilde{Sq}^0_*)^s: (QP_2)_n \to (QP_2)_1$ 
is a isomorphism of  $GL_2$-modules, we need only to compute $(QP_2)_1^{GL_2}$.

According to Peterson \cite{pe}, $(QP_2)_n$ is the vector space of dimension 2 with a basis consisting of 2 classes represented by the following monomials:
$$v_{s,1} = x_1^{2^s-1}x_2^{2^{s+1}-1},\ v_{s,2} = x_1^{2^{s+1}-1}x_2^{2^{s}-1}.$$
In particular, $v_{0,1} = x_2$. $v_{0,2} = x_1$. Suppose $\theta = a_1v_1 +a_2v_2 = a_1x_2 + a_2x_1 \in (QP_2)_1^{GL_2}$ with $a_1, a_2 \in \mathbb F_2$. Then $\rho_1(\theta) = a_1v_2 + a_2v_1 \equiv \theta$.  So, we get $a_1 = a_2$. Since $\rho_2(\theta) \equiv a_1v_1+ a_2(v_1+v_2) \equiv \theta$, we obtain $a_1 = a_2 =0$. Hence,  $(QP_2)_1^{GL_2} = 0$ and $(QP_2)_n^{GL_2} = 0$.

Now, we consider the degree $n = 2^{s+t} + 2^s -2$ with $s, t$ non-negative integers, $t\geqslant 2$. Since  $(\widetilde{Sq}^0_*)^s: (QP_2)_n \to (QP_2)_{2^t-1}$ is a isomorphism of  $GL_2$-modules, we need only to compute $(QP_2)_{2^t-1}^{GL_2}$. 
According to Peterson \cite{pe}, $(QP_2)_{2^t-1}$ is the vector space of dimension 3 with a basis consisting of 3 classes represented by the following monomials:
$$u_{t,1} = x_1^{2^t-1},\ u_{t,2} = x_2^{2^{t}-1},\ u_{t,3} = x_1x_2^{2^{t}-2}.$$
Suppose $\theta_t = a_1u_{t,1} +a_2u_{t,2} +a_3u_{t,3}$ with $a_1, a_2, a_3 \in \mathbb F_2$ and $[\theta_t] \in (QP_2)_{2^t-1}^{GL_2}$. By a simple computation, we have
  $\rho_1(\theta_t) = a_1u_{t,2} + a_2u_{t,1} +a_3u_{t,3} \equiv\theta_t$, hence $a_1 = a_2 = a$. Then, $\rho_2(\theta_t) \equiv a(u_{t,1}+u_{t,2}) + au_{t,2}+ a_3(u_{t,2} + u_{t,3}) \equiv \theta_t$. So, we get $a_3 =a$. Hence, $\theta_t= ap_{2,0,t}$ with $p_{2,0,t} = u_{t,1} + u_{t,2} + u_{t,3}$ and
$$(QP_2)_{n}^{GL_2} = \langle [\psi^s(p_{2,0,t})]\rangle.$$

Combining the above results, we obtain

\begin{props}\label{md421} Let $n$ be a non-negative integer. We have
$$(QP_2)_{n}^{GL_2} = \begin{cases}\langle [p_{2,s}] \rangle, & \text{if }\ n = 2^{s+1}-2,\ s \geqslant 0\\
\langle [p_{2,s,t}]\rangle, & \text{if }\ n = 2^{s+t}+2^s-2,\ s\geqslant 0, \ t\geqslant 2,\\
0, & \text{otherwise},\\
  \end{cases} $$
where $p_{2,s,t} = \psi^s(p_{2,0,t})$.
\end{props}

Recall that $(P_2)^* = \Gamma(a_1,a_2)$. For any $s, t \geqslant 0$,  we set
 $$q_{2,s,t} := a_1^{(2^s-1)}a_2^{(2^{s+t}-1)}\in P((P_2)^*_{2^{s+t} + 2^s-2}).$$ 
Since $\langle q_{2,s,0}, p_{2,s}\rangle = 1$ and $\langle q_{2,s,t}, p_{2,s,t}\rangle = 1$ for every $s \geqslant 0, t\geqslant 2$, from Proposition \ref{md421}, we get the following.
\begin{props} For $n$ a non-negative integer, we obtain
$$\mathbb F_2{\otimes}_{GL_2} P((P_2)^*_n)  = \begin{cases}\langle [q_{2,s,0}] \rangle, & \text{if }\ n = 2^{s+1}-2,\ s \geqslant 0\\
\langle [q_{2,s,t}]\rangle, & \text{if }\ n = 2^{s+t}+2^s-2, s\geqslant 0, \ t\geqslant 2,\\
0, & \text{otherwise}.\\
  \end{cases} $$
\end{props}

It is easy to see that
$\phi_2(q_{2,s,t}) = \lambda_{2^s-1}\lambda_{2^{s+t}-1}$
is a cycle in $\Lambda^{2,*}$. Applying Theorem \ref{md52}, we get
$$Tr_2([q_{2,s,t}]) = [\phi_2(q_{2,s,t})] = [\lambda_{2^s-1}\lambda_{s+t}]= h_sh_{s+t}.$$

Since $h_sh_{s+1} =0$, applying Theorem \ref{md51}, we have 
$$\text{Ext}_{\mathcal A}^{2, m}(\mathbb F_2, \mathbb F_2)  = \begin{cases} 
\langle h_s^2\rangle, &\text{if } m = 2^{s+1}, \text{ with } s\geqslant 0,\\
\langle h_sh_{s+t}\rangle, &\text{if } m = 2^{s+t}+2^s, \text{ with } s\geqslant 0, t \geqslant 2,\\
0, &\text{otherwise}. 
\end{cases}$$
Theorem \ref{dls} is completely proved.

\subsection{Determination of $Tr_3$}\

\medskip
In this subsection, we present the proof of the following.
\begin{thms}[Boardman \cite{bo}]\label{dlb} The third Singer algebraic transfer 
$$Tr_3 :\mathbb F_2{\otimes}_{GL_3} P((P_3)^*) \longrightarrow \text{\rm Ext}_{\mathcal A}^{3,*+3}(\mathbb F_2,\mathbb F_2)$$ 
is an isomorphism.
\end{thms}

To prove this theorem, Boardman \cite{bo} computed the space $QP_3^{GL_3}$ by using a basis consisting of the all the classes represented by certain polynomials in $P_3$. It is difficult to use his method for $k = 4$, where there are 315 polynomials instead of 21. We also compute this space, however we use the admissible monomial basis for $QP_3$ that is different from the one of Boardman in \cite{bo}. Our approach can be apply for $k = 4$ by using the admissible monomial basis for $QP_4$ which is given in \cite{su2,su5}. 

\medskip
From a result of Wood \cite{wo}, we need only to compute $QP_3^{GL_3}$ in the degree $n$ with $\mu(n) \leqslant 3$. 

\subsubsection {The case $n = 2^{t+1}-2$}\label{bab1}\

\medskip
According to Kameko \cite{ka}, $(QP_3)_n$ is a vector space with a basis consisting of all the classes represented by the following monomials:
\begin{align*}&v_{t,1} = x_2^{2^t-1}x_3^{2^t-1}, \ v_{t,2} = x_1^{2^t-1}x_3^{2^t-1},\  v_{t,3} = x_1^{2^t-1}x_2^{2^t-1},\ \text{ for $t\geqslant 1$},\\
&v_{t,4} = x_1x_2^{2^t-2}x_3^{2^t-1},\  v_{t,5} = x_1x_2^{2^t-1}x_3^{2^t-2},\ v_{t,6} = x_1^{2^t-1}x_2x_3^{2^t-2}, \ \text{ for $t\geqslant 2$},\\
&v_{t,7} = x_1^3x_2^{2^t-3}x_3^{2^t-2}, \ \text{ for $t\geqslant 3$}
\end{align*}

Set $p_{3,t} = \sum_{i=1}^7v_{t,i}$, with  $t \geqslant 3$. By a direct computation, we have
\begin{props}\label{dl312} For any non-negative integer $t$, we have
$$ (QP_3)_{2^{t+1}-2}^{GL_3} = \begin{cases} \langle 1 \rangle, &\text{if }\ t=0,\\
0, &\text{if }\ t=1,2,\\
\langle [p_{3,t} ] \rangle, &\text{if }\ t\geqslant 3.\end{cases}
$$
\end{props}

Recall that $(P_3)^* = \Gamma(a_1,a_2, a_3)$. We set 
$$q_{3,t} = a_1^{(0)}a_2^{(2^t-1)}a_3^{(2^{t}-1)}\in P((P_3)^*_{2^{t+1}-2}).$$
Since $\langle p_{3,t},q_{3,t}\rangle =1$, we get 
$$\mathbb F_2{\otimes}_{GL_3} P((P_3)^*_{2^{t+1} - 2})  = \begin{cases}\langle [1] \rangle, & \text{if }\ t=0\\ 0, & \text{if }\ t=1,2,\\
\langle [q_{3,t}]\rangle, & \text{if }\  t\geqslant 3.\\
\end{cases} $$

It is easy to see that $\phi_3(q_{3,t}) = \lambda_0\lambda_{2^t-1}^2$ 
is a cycle in $\Lambda^{3,*}$. By Theorem \ref{md52}, we have
$$Tr_3([q_{3,t}]) = [\phi_3(q_{3,t})] = [\lambda_{0}\lambda_{2^t-1}^2]= h_0h_{t}^2.$$

According to Theorem \ref{md51}, we have 
$$\text{Ext}_{\mathcal A}^{3, 2^{t+1}+1}(\mathbb F_2, \mathbb F_2)  =  \langle h_0h_{t}^2\rangle.$$
Since $h_0h_{1} =0$ and $h_0h_2^2 =0$, from the above equalities we see that Theorem \ref{dlb} is true in this case.

\subsubsection {The case $n = 2^{t+u}+ 2^u - 3$}\label{bab2}\
 
\medskip
If $u>1$ then $\mu(n) =3$, hence the iterated Kameko homomorphism 
$$(\widetilde{Sq}^0_*)^{u-1} :(QP_3)_{2^{t+u}+ 2^u - 3} \to (QP_3)_{2^{t+1}-1}$$ 
is also an isomorphism $GL_3$-modules. Hence, we need only to compute $(QP_3)_{2^{t+1}-1}^{GL_3}$. According to Kameko \cite{ka}, $(QP_3)_{n}$ is a vector space with a basis consisting of all the classes represented by the following monomials:
\begin{align*}&u_{t,1} = x_3^{2^{t+1} - 1}, \ u_{t,2} = x_2^{2^{t+1} - 1},\  u_{t,3} = x_1^{2^{t+1} - 1},\ \text{ for } \ t\geqslant 0,\\
&u_{t,4} = x_2x_3^{2^{t+1} - 2},\  u_{t,5} = x_1x_3^{2^{t+1} - 2},\  u_{t,6} = x_1x_2^{2^{t+1} - 2},\ \text{ for } \ t\geqslant 1,\\
&u_{1,7} = x_1x_2x_3, \text{ for } \ t = 1,\\
&u_{t,7} = x_1x_2^2x_3^{2^{t+1}-4}, \ u_{t,8} = x_1x_2^{2^t-1}x_3^{2^t-1},\\ 
&u_{t,9} = x_1^{2^t-1}x_2x_3^{2^t-1},\ u_{t,10} = x_1^{2^t-1}x_2^{2^t-1} x_3,\ \text{ for } \ t\geqslant 2,\\
&u_{t,11} = x_1^3x_2^{2^t-3}x_3^{2^t-1},\  u_{t,12} = x_1^3x_2^{2^t-1}x_3^{2^t-3}, \  u_{t,13} = x_1^{2^t-1}x_2^3x_3^{2^t-3},\ \text{ for } \ t\geqslant 3,\\
&u_{t,14} =  x_1^7x_2^{2^t-5}x_3^{2^t-3},\ \text{ for } \ t\geqslant 4.
\end{align*}

Set $p_{3,t,1} =  \sum_{i=1}^7u_{t,i}$ for $t \geqslant 1$ and $\bar p_{3,t,1} = \sum_{j=7}^{14}u_{t,j}$ for $t \geqslant 4$. By a direct computation we have
\begin{props}\label{mdp42} For any integers $t\geqslant 0,\ u>0$, we have
$$ (QP_3)_{2^{t+u}+2^u-3}^{GL_3} = \begin{cases} 
0, &\text{if }\ t=0,\\
\langle [p_{3,t,u}] \rangle, &\text{if }\ 1 \leqslant t \leqslant 3,\\
\langle [p_{3,t,u}], [\bar p_{3,t,u}] \rangle, &\text{if }\ t\geqslant 4,\end{cases}
$$
where $p_{3,t,u} = \psi^{u-1}(p_{3,t,1})$,  $\bar p_{3,t,u} = \psi^{u-1}(\bar p_{3,t,1})$.
\end{props}

We set
$$q_{3,t,u} = a_1^{(2^{u-1}-1)}a_2^{(2^{u-1}-1)}a_3^{(2^{t+u}-1)},\quad \bar q_{3,t,u} = a_1^{(2^u-1)}a_2^{(2^{t+u-1}-1)}a_3^{(2^{t+u-1}-1)}. $$
It is easy to see that  $q_{3,t,u},\ \bar q_{3,t,u}  \in  P((P_3)^*_{2^{t+1}-2})$ and 
\begin{align*}&\langle p_{3,t,u},q_{3,t,u}\rangle =1, \ \langle p_{3,t,u},\bar q_{3,t,u}\rangle =0, \\
&\langle \bar p_{3,t,u},q_{3,t,u}\rangle =0, \ \langle \bar p_{3,t,u},\bar q_{3,t,u}\rangle =1.
\end{align*}
So, we get 
$$\mathbb F_2{\otimes}_{GL_3} P((P_3)^*_{2^{t+u} + 2^u - 3})  = \begin{cases}
0,& \text{if }\ t=0,\\
\langle [q_{3,t,u}] \rangle, & \text{if }\ 1 \leqslant t \leqslant 3,\\
\langle [q_{3,t,u}] ,[\bar q_{3,t,u}] \rangle, & \text{if }\  t\geqslant 4.
\end{cases} $$

By applying Theorem \ref{md52}, we have  
\begin{align*}\phi_3(q_{3,t,u}) &= \lambda_{2^{u-1}-1}^2\lambda_{2^{t+u}-1},\\
\phi_3(\bar q_{3,t,u}) &= \lambda_{2^u-1}\lambda_{2^{t+u-1}-1}^2
\end{align*}
are the cycles in $\Lambda^{3,*}$. So, we obtain
\begin{align*}Tr_3([q_{3,t,u}]) &= [\phi_3(q_{3,t,u})] = [\lambda_{2^{u-1}-1}^2\lambda_{2^{t+u}-1}]= h_{u-1}^2h_{t+u},\\ Tr_3([\bar q_{3,t,u}]) &= [\phi_3(\bar q_{3,t,u})] = [\lambda_{2^{u}-1}^2\lambda_{2^{t+u-1}-1}^2]= h_{u}h_{t+u-1}^2.
\end{align*}

According to Theorem \ref{md51}, we have 
$$\text{Ext}_{\mathcal A}^{3, 2^{t+u}+2^u}(\mathbb F_2, \mathbb F_2)  =  \langle h_uh_{t+u-1}^2, h_{u-1}^2h_{t+u} \rangle.$$
If $t=0$ then  $h_uh_{u-1}^2 = h_u^2h_{u-1} =0$. If $t = 1$ then $h_uh_{t+u-1}^2 = h_u^3 = h_{u-1}^2h_{u+1} = h_{u-1}^2h_{t+u}$. If $t=2$ then $h_uh_{t+u-1}^2 = h_uh_{u+1}^2 =0$. If $t=3$ then $h_uh_{t+u-1}^2 = h_uh_{u+2}^2 =0$. Hence, from the above equalities we can easily see that Theorem \ref{dlb} is true in this case.

\subsubsection {The case $n = 2^{s+u+1}+ 2^{u+1} +2^u - 3$}\label{bab3}\

\medskip
If $u>0$ then $\mu(n) =3$, hence the iterated Kameko homomorphism 
$$(\widetilde{Sq}^0_*)^{u} :(QP_3)_{2^{s+u}+ 2^u - 3} \to (QP_3)_{2^{s+1}}$$ 
is also an isomorphism of $GL_3$-modules. Hence, we need only to compute $(QP_3)_{2^{s+1}}^{GL_3}$. 

According to Kameko \cite{ka}, $(QP_3)_{2^{s + 1}}$ is a vector space with a basis consisting of all the classes represented by the following monomials:
\begin{align*}
&v_{s,1} = x_2x_3^{2^{s + 1}-1},\ v_{s,2} = x_2^{2^{s + 1}-1}x_3,\ v_{s,3} = x_1x_3^{2^{s + 1}-1},\\ &v_{s,4} = x_1x_2^{2^{s + 1}-1},\ v_{s,5} = x_1^{2^{s + 1}-1}x_3,\ v_{s,6} = x_1^{2^{s + 1}-1}x_2, \ \text{for }\ s \geqslant 1,\\
&v_{1,7} = x_1x_2x_3^2, \ v_{1,8} = x_1x_2^2x_3,\  \text{for }\ s=1,\\
&v_{s,7} = x_2^3x_3^{2^{s + 1}-3},\ v_{s,8} = x_1^3x_3^{2^{s + 1}-3},\ v_{s,9} = x_1^3x_2^{2^{s + 1}-3}, \\ &v_{s,10} = x_1x_2x_2^{2^{s + 1}-2},\ v_{s,11} = x_1x_2^{2^{s + 1}-2}x_3,\ v_{s,12} =x_1x_2^2x_3^{2^{s + 1}-3}, \\ &v_{s,13} =x_1x_2^3x_3^{2^{s + 1}-4},\ v_{s,14} =x_1^3x_2x_3^{2^{s + 1}-4}\ \text{ for }\ s \geqslant 2\\
&v_{15} = x_1^3x_2^4x_3, \  \text{for }\ s = 2.
\end{align*}

Set $\bar p_0 =  v_{2,10}+v_{2,11}+v_{2,14}+v_{2,15}$. By a direct computation, we have
\begin{props}\label{mdp43} For any integers $s> 0, u \geqslant 0$ and $ n= 2^{s+u+1} + 2^{u+1} + 2^u -3$, we have
$$ (QP_3)_{n}^{GL_3} = 
\begin{cases}\langle [\psi^u(\bar p_0)]\rangle, &\text{if } s=2,\\ 0, &\text{if } s \ne 2.\end{cases}$$
\end{props}
 We set 
\begin{align*}\bar c_u &= a_1^{(3.2^u-1)}a_2^{(4.2^u-1)}a_3^{(4.2^u-1)} + a_1^{(2.2^u-1)}a_2^{(5.2^u-1)}a_3^{(4.2^u-1)}\\
&\quad + a_1^{(2.2^u-1)}a_2^{(3.2^u-1)}a_3^{(6.2^u-1)} + a_1^{(2.2^u-1)}a_2^{(2.2^u-1)}a_3^{(7.2^u-1)} \end{align*}
is an element in $(P_3)^* = \Gamma(a_1,a_2,a_3)$. By a direct computation, we can see that 
$\bar c_u \in P((P_3)^*_{2^{t+u} +2^u - 3})$ and $\langle \psi^u(\bar p_0),\bar c_u\rangle =1$. So, we get
 $$ \mathbb F_2{\otimes}_{GL_3} P((P_3)^*_{n}) =
\begin{cases}\langle [\bar c_u]\rangle, &\text{if } s=2,\\ 0, &\text{if } s \ne 2.\end{cases}$$
For $u=0$, we have $\bar c_0 =  a_1^{(2)}a_2^{(3)}a_3^{(3)} + a_1^{(1)}a_2^{(4)}a_3^{(3)} + a_1^{(1)}a_2^{(2)}a_3^{(5)} + a_1^{(1)}a_2^{(1)}a_3^{(6)}$. 

A direct computation shows 
\begin{align*} \phi_3( a_1^{(2)}a_2^{(3)}a_3^{(3)}) &= \lambda_2\lambda_3^2 + \lambda_1\lambda_4\lambda_3 + \lambda_1\lambda_3\lambda_4,\\
\phi_3( a_1^{(1)}a_2^{(4)}a_3^{(3)}) &= \lambda_1\lambda_4\lambda_3 + \lambda_1\lambda_3\lambda_4 + \lambda_1\lambda_2\lambda_5,\\
\phi_3( a_1^{(1)}a_2^{(2)}a_3^{(5)}) &= \lambda_1\lambda_2\lambda_5+ \lambda_1^2\lambda_6 ,\\
\phi_3( a_1^{(1)}a_2^{(1)}a_3^{(6)}) &=  \lambda_1^2\lambda_6.
\end{align*}
Hence, we obtain 
$\phi_3(\bar c_0) = \lambda_2\lambda_3^2$. 
By Theorem \ref{md52}, we have $Tr_3([\bar c_0]) = [\lambda_2\lambda_3^2] = c_0$.
Since $[\bar c_u] = (\widetilde{Sq}_*^0)^u([\bar c_0])$, we get
$$Tr_3([\bar c_u]) = Tr_3((\widetilde{Sq}^0)^u([\bar c_0])) = ({Sq}^0)^u Tr_3([\bar c_0]) = ({Sq}^0)^u (c_0) = c_u.
$$
By Theorem \ref{md51}, we have $h_uh_{u+1}=0$. Hence,
$$\text{Ext}_{\mathcal A}^{3, 2^{s+u+1}+2^{u+1}+2^u}(\mathbb F_2, \mathbb F_2)  =  \begin{cases}\langle h_uh_{u+1}h_{u+3}, c_u \rangle= \langle  c_u \rangle, &\text{if }\ s=2,\\
\langle h_uh_{u+1}h_{s+u+1} \rangle = 0, &\text{if }\ s\ne 2.\end{cases}
$$

Theorem \ref{dlb} in this case follows from the above equalities.

\subsubsection {The case of the generic degree}\label{bab4}\

\medskip
 In this subsection, we consider the degree
$$n = 2^{s+t+u} + 2^{t+u} + 2^u -3,$$
with $s,t,u$ non-negative integers.

The subcases either $s=0$ or $t=0$ have been determined in Subsections  \ref{bab1} and \ref{bab2}. The case $s>0$ and $t=1$ has been determined in Subsection \ref{bab3}. So, we assume that $s>0$ and $t>1$. 

The iterated homomorphism
$$(\widetilde{Sq}^0_*)^{u} :(QP_3)_{2^{s+t+u}+ 2^{t+u} + 2^u - 3} \to (QP_3)_{2^{s+t}+2^t-2}$$  
is an isomorphism of $GL_3$-modules. So, we need only to compute $(QP_3)_{2^{s+t}+2^t-2}^{GL_3}$.

The subcase $s=1$. Then $n = 2^{t+1} + 2^{t}-2$.
 According to Kameko \cite{ka}, $(QP_3)_{n}$ is the vector space with a basis consisting of all the classes represented by the following monomials:

\medskip
\centerline{\begin{tabular}{lll}
$v_{t,1} = x_2^{2^t-1}x_3^{2^{t+1}-1}$& $v_{t,2} = x_2^{2^{t+1}-1}x_3^{2^t-1}$& $v_{t,3} = x_1^{2^t-1}x_3^{2^{t+1}-1}$\cr 
$v_{t,4} = x_1^{2^t-1}x_2^{2^{t+1}-1}$& $v_{t,5} = x_1^{2^{t+1}-1}x_3^{2^t-1}$& $v_{t,6} = x_1^{2^{t+1}-1}x_2^{2^t-1}.$\cr
$v_{t,7} = x_1x_2^{2^{t}-2}x_3^{2^{t+1}-1}$& $v_{t,8} = x_1x_2^{2^{t+1}-1}x_3^{2^{t}-2}$&$v_{t,9} = x_1^{2^{t+1}-1}x_2x_3^{2^{t}-2}$\cr
$v_{t,10} = x_1x_2^{2^{t}-1}x_3^{2^{t+1}-2}$& $v_{t,11} = x_1x_2^{2^{t+1}-2}x_3^{2^{t}-1}$&$v_{t,12} =x_1^{2^{t}-1} x_2x_3^{2^{t+1}-2}$\cr
$v_{t,13} = x_1^3x_2^{2^{t+1}-3}x_3^{2^{t}-2}$,&&\cr
\end{tabular}}

 \noindent  $v_{2,14} = x_1^3x_2^3x_3^4$ for $t=2$,  and \ $v_{t,14} = x_1^3x_2^{2^{t}-3}x_3^{2^{t+1}-2}$ for  $t > 2.$

\medskip
By a direct computation using the above basis, we obtain
\begin{props}\label{mdp442} For any integers $t> 1, u \geqslant 0$ and $ n= 2^{t+u+1} + 2^{t+u} + 2^u -3$, we have
$ (QP_3)_{n}^{GL_3} = 0.$
\end{props}
By Theorem \ref{md51}  $h_{t+u}h_{t+u+1} =0$, so we have 
$$\text{Ext}_{\mathcal A}^{3, 2^{t+u+1}+2^{t+u}+2^u}(\mathbb F_2, \mathbb F_2)  =  \langle h_uh_{t+u}h_{t+u+1} \rangle =0.$$ 
Hence, from the above equalities, we can see that
$$Tr_3 :\mathbb F_2{\otimes}_{GL_3} P((P_3)^*_{2^{t+u+1} +2^{t+u}+2^u - 3}) \longrightarrow \text{\rm Ext}_{\mathcal A}^{3,2^{t+u+1}+2^{t+u} + 2^u}(\mathbb F_2,\mathbb F_2)$$ 
is a trivial isomorphism.

Now, suppose that $s,t>1$ and $n = 2^{s+t} + 2^t - 2$. From the results of Kameko \cite{ka}, we see that
 $(QP_3)_{n}$  is the vector space of dimension 21 with a basis consisting of all the classes represented by the following monomials:

\medskip
\centerline{\begin{tabular}{ll}
$v_{s,t,1} = x_2^{2^t - 1}x_3^{2^{s+t} - 1}$&
$v_{s,t,2} = x_2^{2^{s+t} - 1}x_3^{2^t - 1}$\cr
$v_{s,t,3} = x_1^{2^t - 1}x_3^{2^{s+t} - 1}$&
$v_{s,t,4} = x_1^{2^t - 1}x_2^{2^{s+t} - 1}$\cr
$v_{s,t,5} = x_1^{2^{s+t} - 1}x_3^{2^t - 1}$&
$v_{s,t,6} = x_1^{2^{s+t} - 1}x_2^{2^t - 1}$\cr
$v_{s,t,7} = x_2^{2^{t+1} - 1}x_3^{2^{s+t} - 2^t - 1}$&
$v_{s,t,8} = x_1^{2^{t+1} - 1}x_3^{2^{s+t} - 2^t - 1}$\cr
$v_{s,t,9} = x_1^{2^{t+1} - 1}x_2^{2^{s+t} - 2^t - 1}$&
$v_{s,t,10} = x_1x_2^{2^t - 2}x_3^{2^{s+t} - 1}$\cr
$v_{s,t,11} = x_1x_2^{2^{s+t} - 1}x_3^{2^t - 2}$&
$v_{s,t,12} = x_1^{2^{s+t} - 1}x_2x_3^{2^t - 2}$\cr
$v_{s,t,13} = x_1x_2^{2^t - 1}x_3^{2^{s+t} - 2}$&
$v_{s,t,14} = x_1x_2^{2^{s+t} - 2}x_3^{2^t - 1}$\cr
$v_{s,t,15} = x_1^{2^t - 1}x_2x_3^{2^{s+t} - 2}$&
$v_{s,t,16} = x_1x_2^{2^{t+1} - 2}x_3^{2^{s+t} - 2^t - 1}$\cr
$v_{s,t,17} = x_1x_2^{2^{t+1} - 1}x_3^{2^{s+t} - 2^t - 2}$&
$v_{s,t,18} = x_1^{2^{t+1} - 1}x_2x_3^{2^{s+t} - 2^t - 2}$\cr
$v_{s,t,19} = x_1^{3}x_2^{2^{s+t} - 3}x_3^{2^t - 2}$&
$v_{s,t,20} = x_1^{3}x_2^{2^{t+1} - 3}x_3^{2^{s+t} - 2^t - 2}$,\cr
\end{tabular}}

\medskip\noindent
$v_{s,2,21} = x_1^3x_2^3x_3^{2^{s+2}-4}$, for  $t=2$\ and $v_{s,t,21} = x_1^{3}x_2^{2^t - 3}x_3^{2^{s+t} - 2}$ for  $t > 2.$

\medskip
We set 
$$p_{3,s,t,u}=\begin{cases}\sum_{1\leqslant j \leqslant 21, j\ne 13,15}\psi^u(v_{s,2,j}),& \text{ if } t=2,\\ 
\sum_{1\leqslant j \leqslant 21}\psi^u(v_{s,t,j}),& \text{ if } t>2.
\end{cases} $$
By a direct computation using this basis, we get
\begin{props}\label{mdp444} For any integers $s,t > 1, u \geqslant 0$ and $ n= 2^{s+ t+u} + 2^{t+u} + 2^u -3$, we have
$ (QP_3)_{n}^{GL_3} = \langle[p_{3,s,t,u}]\rangle.$
\end{props}

By Theorem \ref{md51}, we have 
$$\text{Ext}_{\mathcal A}^{3, 2^{s+t+u}+2^{t+u}+2^u}(\mathbb F_2, \mathbb F_2)  =  \langle h_uh_{t+u}h_{s+t+u} \rangle.$$
Note that $\psi^u(v_{s,t,1}) = x_1^{2^u-1}x_2^{2^{t+u}-1}x_3^{2^{s+t+u}-1}$. Consider the element 
$$q_{3,s,t,u} = a_1^{(2^u-1)}a_2^{(2^{t+u}-1)}a_3^{(2^{s+t+u}-1)} \in \mathbb F_2{\otimes}_{GL_3} P((P_3)^*_{n}).$$
Since $\langle p_{3,s,t,u},q_{3,s,t,u}\rangle =1$, from Proposition \ref{mdp444}, we obtain 
$$\mathbb F_2{\otimes}_{GL_3} P((P_3)^*_{n}) = \langle [q_{3,s,t,u}]\rangle .$$

It is easy to see that $\phi_3(q_{3,s,t,u}) = \lambda_u\lambda_{t+u}\lambda_{s+t+u}$, hence using Theorem \ref{md52} we get 
$$Tr_3 ([q_{3,s,t,u}]) = [\lambda_u\lambda_{t+u}\lambda_{s+t+u}] = h_uh_{t+u}h_{s+t+u}.$$
Theorem \ref{dlb} is completely proved.

%========================================
\section{Determination of $Tr_4$ in some internal degrees}\label{s4}

In this section, we explicitly determined $Tr_4$ in some internal degrees. Our main result is the following.

\begin{thm}\label{dlm} Let $s$ be a positive integer and let $n$ be one of the degrees $2^{s+1}- 1$, $2^{s+1}- 2$, $2^{s+1}- 3$. If $n \ne 61$ and $n \ne 126$, then the homomorphism
$$Tr_4 :\mathbb F_2{\otimes}_{GL_4} P((P_4)_n^*) \longrightarrow \text{\rm Ext}_{\mathcal A}^{4,n+4}(\mathbb F_2,\mathbb F_2)$$ 
is an isomorphism. If either $n = 61$ or $n = 126$, then $Tr_4$ is a monomorphism but it is not an epimorphism.
\end{thm}
 We prove the theorem by computing the space $(QP_4)_n^{GL_4}$. 

\subsection{The case $n = 2^{s+1}-3$}

\begin{props}[see \cite{su1,su5}]\label{mdd41}Let $n = 2^{s+1}-3$ with $s$ a positive integer. Then, the dimension of the $\mathbb F_2$-vector space $(QP_4)_{n}$ is determined by the following table:

\medskip
\centerline{\begin{tabular}{c|cccc}
$n = 2^{s+1}-3$&$s=1$ & $s=2$ & $s=3$& $s\geqslant 4$\cr
\hline
\ $\dim(QP_4)_n$ & $4$ & $15$ & $35$ &$45$ \cr
\end{tabular}}
\end{props}
A basis for $(QP_4)_n$ is the set consisting of all the classes represented monomials $a_j= a_{s,j}$ which are determined as follows:

\medskip
For $s = 1$, $a_{1,1} = x_4,\ a_{1,2} = x_3,\ a_{1,3} = x_2, a_{1,4} = x_1$.

\medskip
For $s \geqslant 2$, 

\medskip
 \centerline{\begin{tabular}{lll}
$a_{s,1} = x_2^{2^{s-1}-1}x_3^{2^{s-1}-1}x_4^{2^s-1}$ &  $a_{s,2} = x_2^{2^{s-1}-1}x_3^{2^s-1}x_4^{2^{s-1}-1}$\cr  
$a_{s,3} = x_2^{2^s-1}x_3^{2^{s-1}-1}x_4^{2^{s-1}-1}$ &  $a_{s,4} = x_1^{2^{s-1}-1}x_3^{2^{s-1}-1}x_4^{2^s-1}$\cr  
$a_{s,5} = x_1^{2^{s-1}-1}x_3^{2^s-1}x_4^{2^{s-1}-1}$ &  $a_{s,6} = x_1^{2^{s-1}-1}x_2^{2^{s-1}-1}x_4^{2^s-1}$\cr  $a_{s,7} = x_1^{2^{s-1}-1}x_2^{2^{s-1}-1}x_3^{2^s-1}$ &  $a_{s,8} = x_1^{2^{s-1}-1}x_2^{2^s-1}x_4^{2^{s-1}-1}$\cr  
$a_{s,9} = x_1^{2^{s-1}-1}x_2^{2^s-1}x_3^{2^{s-1}-1}$ &  $a_{s,10} = x_1^{2^s-1}x_3^{2^{s-1}-1}x_4^{2^{s-1}-1}$\cr  
$a_{s,11} = x_1^{2^s-1}x_2^{2^{s-1}-1}x_4^{2^{s-1}-1}$ &  $a_{s,12} = x_1^{2^s-1}x_2^{2^{s-1}-1}x_3^{2^{s-1}-1}$\cr  
\end{tabular}}

\medskip
For $s = 2$, $a_{2,13} = x_1x_2x_3x_4^{2}$,\  $a_{2,14} = x_1x_2x_3^{2}x_4$,\  $a_{2,15} = x_1x_2^{2}x_3x_4$. 

\medskip
For $s \geqslant 3$,

\medskip
 \centerline{\begin{tabular}{ll}
$a_{s,13} = x_1x_2^{2^{s-1}-2}x_3^{2^{s-1}-1}x_4^{2^s-1}$ &  $a_{s,14} = x_1x_2^{2^{s-1}-2}x_3^{2^s-1}x_4^{2^{s-1}-1}$\cr  
$a_{s,15} = x_1x_2^{2^{s-1}-1}x_3^{2^{s-1}-2}x_4^{2^s-1}$ &  $a_{s,16} = x_1x_2^{2^{s-1}-1}x_3^{2^s-1}x_4^{2^{s-1}-2}$\cr  
$a_{s,17} = x_1x_2^{2^s-1}x_3^{2^{s-1}-2}x_4^{2^{s-1}-1}$ &  $a_{s,18} = x_1x_2^{2^s-1}x_3^{2^{s-1}-1}x_4^{2^{s-1}-2}$\cr  
$a_{s,19} = x_1^{2^{s-1}-1}x_2x_3^{2^{s-1}-2}x_4^{2^s-1}$ &  $a_{s,20} = x_1^{2^{s-1}-1}x_2x_3^{2^s-1}x_4^{2^{s-1}-2}$\cr  
$a_{s,21} = x_1^{2^{s-1}-1}x_2^{2^s-1}x_3x_4^{2^{s-1}-2}$ &  $a_{s,22} = x_1^{2^s-1}x_2x_3^{2^{s-1}-2}x_4^{2^{s-1}-1}$\cr  
$a_{s,23} = x_1^{2^s-1}x_2x_3^{2^{s-1}-1}x_4^{2^{s-1}-2}$ &  $a_{s,24} = x_1^{2^s-1}x_2^{2^{s-1}-1}x_3x_4^{2^{s-1}-2}$\cr  
$a_{s,25} = x_1x_2^{2^{s-1}-1}x_3^{2^{s-1}-1}x_4^{2^s-2}$ &  $a_{s,26} = x_1x_2^{2^{s-1}-1}x_3^{2^s-2}x_4^{2^{s-1}-1}$\cr  
$a_{s,27} = x_1x_2^{2^s-2}x_3^{2^{s-1}-1}x_4^{2^{s-1}-1}$ &  $a_{s,28} = x_1^{2^{s-1}-1}x_2x_3^{2^{s-1}-1}x_4^{2^s-2}$\cr  
$a_{s,29} = x_1^{2^{s-1}-1}x_2x_3^{2^s-2}x_4^{2^{s-1}-1}$ &  $a_{s,30} = x_1^{2^{s-1}-1}x_2^{2^{s-1}-1}x_3x_4^{2^s-2}$\cr  
\end{tabular}}

\medskip
For $s = 3$, 

\medskip
 \centerline{\begin{tabular}{lll}
$a_{3,31} = x_1^{3}x_2^{3}x_3^{5}x_4^{2}$ &  $a_{3,32} = x_1^{3}x_2^{5}x_3^{2}x_4^{3}$ &  $a_{3,33} = x_1^{3}x_2^{5}x_3^{3}x_4^{2}$\cr
$a_{3,34} = x_1^{3}x_2^{3}x_3^{3}x_4^{4}$&  $a_{3,35} = x_1^{3}x_2^{3}x_3^{4}x_4^{3}$ &  \cr 
\end{tabular}}

\medskip
For $s \geqslant 4$,

\medskip
 \centerline{\begin{tabular}{ll}
$a_{s,31} = x_1^{3}x_2^{2^{s-1}-3}x_3^{2^{s-1}-2}x_4^{2^s-1}$ &  $a_{s,32} = x_1^{3}x_2^{2^{s-1}-3}x_3^{2^s-1}x_4^{2^{s-1}-2}$\cr  
$a_{s,33} = x_1^{3}x_2^{2^s-1}x_3^{2^{s-1}-3}x_4^{2^{s-1}-2}$ &  $a_{s,34} = x_1^{2^s-1}x_2^{3}x_3^{2^{s-1}-3}x_4^{2^{s-1}-2}$\cr  
$a_{s,35} = x_1^{3}x_2^{2^{s-1}-3}x_3^{2^{s-1}-1}x_4^{2^s-2}$ &  $a_{s,36} = x_1^{3}x_2^{2^{s-1}-3}x_3^{2^s-2}x_4^{2^{s-1}-1}$\cr  
$a_{s,37} = x_1^{3}x_2^{2^{s-1}-1}x_3^{2^{s-1}-3}x_4^{2^s-2}$ &  $a_{s,38} = x_1^{2^{s-1}-1}x_2^{3}x_3^{2^{s-1}-3}x_4^{2^s-2}$\cr  
$a_{s,39} = x_1^{3}x_2^{2^{s-1}-1}x_3^{2^s-3}x_4^{2^{s-1}-2}$ &  $a_{s,40} = x_1^{3}x_2^{2^s-3}x_3^{2^{s-1}-2}x_4^{2^{s-1}-1}$\cr  
$a_{s,41} = x_1^{3}x_2^{2^s-3}x_3^{2^{s-1}-1}x_4^{2^{s-1}-2}$ &  $a_{s,42} = x_1^{2^{s-1}-1}x_2^{3}x_3^{2^s-3}x_4^{2^{s-1}-2}$\cr  
$a_{s,43} = x_1^{7}x_2^{2^s-5}x_3^{2^{s-1}-3}x_4^{2^{s-1}-2}$ & \cr
\end{tabular}}

\medskip
For $s = 4$,\ $a_{4,44} = x_1^{7}x_2^{7}x_3^{9}x_4^{6}$,\  $a_{4,45} = x_1^{7}x_2^{7}x_3^{7}x_4^{8}$.

\medskip
For $s \geqslant 5$, $a_{s,44} = x_1^{7}x_2^{2^{s-1}-5}x_3^{2^s-3}x_4^{2^{s-1}-2}$,\  $a_{s,45} = x_1^{7}x_2^{2^{s-1}-5}x_3^{2^{s-1}-3}x_4^{2^s-2}$.

\medskip
\begin{props}\label{mdd411} Let $s$ be a positive integer. Then,
$(QP_4)_{2^{s+1}-3}^{GL_4} = 0.$
\end{props}
For simplicity, we prove the proposition in detail for $s \geqslant 5$. The other cases can be proved by the similar computations.

 For any monomials $z_1, z_2, \ldots, z_m$ in $P_k$ and for a subgroup $G\subset GL_k$, we denote 
$G(z_1, z_2, \ldots, z_m)$ the $G$-submodule of $QP_k$ generated by the set $\{[z_i] : 1 \leqslant i \leqslant m\}$. We have the following.
\begin{lems}\label{bd2m3} {\rm i)} For any $s\geqslant 5$, there is a direct summand decomposition of the $\Sigma_4$-modules:
$$(QP_4)_{2^{s+1}-3} = \Sigma_4(a_{s,1})\oplus \Sigma_4(a_{s,13}) \oplus \Sigma_4(a_{s,31}) \oplus \Sigma_4(a_{s,25},a_{s,35},a_{s,43}).$$

{\rm ii)} $\Sigma_4(a_{s,1})^{\Sigma_4} = \langle [p_{4,s,1}]\rangle$, with $p_{4,s,1}= \sum_{j=1}^{12}a_{s,j}$.

{\rm iii)} $\Sigma_4(a_{s,13})^{\Sigma_4} = \langle [p_{4,s,2}]\rangle$, with $p_{4,s,2}= \sum_{j=13}^{24}a_{s,j}$.

{\rm iv)} $\Sigma_4(a_{s,31})^{\Sigma_4} =  \langle [p_{4,s,3}]\rangle$, with $p_{4,s,3}= \sum_{j=31}^{34}a_{s,j}$.

{\rm v)} $\Sigma_4(a_{s,25},a_{s,35},a_{s,43})^{\Sigma_4} = \langle [p_{4,s,4}]\rangle$, with 
$$p_{4,s,4}= \sum_{j=25}^{30}a_{s,j} +  \sum_{j=39}^{43}a_{s,j} + a_{s,45}.$$
\end{lems}
\begin{proof} We obtain Part i) by a simple computation using Proposition \ref{mdd41}. We prove Part v) in detail. The others can be proved by the similar computations. By a simple computation we see that the set $\{[a_{s,j}]: j = 25,\ldots,30, 35 \ldots, 45\}$ is a basis for $\Sigma_4(a_{s,25},a_{s,35},a_{s,43})$. Suppose $[f] \in \Sigma_4(a_{s,25},a_{s,35},a_{s,43})^{\Sigma_4}$, then 
$$ f \equiv  \sum_{j=25}^{30}\gamma_ja_{s,j} +  \sum_{j=35}^{45}\gamma_ja_{s,j}$$
with $\gamma_j \in \mathbb F_2$. By a direct computation, we get
\begin{align*}
\rho_1(f) + f &\equiv (\gamma_{25} + \gamma_{28})(a_{s,25} + a_{s,28})+  (\gamma_{26} + \gamma_{29})(a_{s,26} + a_{s,29})\\
&\quad +  (\gamma_{27} + \gamma_{41})a_{s,35} +  (\gamma_{27} + \gamma_{40})a_{s,36}+  (\gamma_{37} + \gamma_{38})(a_{s,37} + a_{s,38})\\
&\quad  +  (\gamma_{39} + \gamma_{42})(a_{s,39}+ a_{s,42}) +  (\gamma_{41} + \gamma_{43})a_{s,44} +  (\gamma_{40} + \gamma_{43})a_{s,45} \equiv 0,\\
\rho_2(f) + f &\equiv (\gamma_{26} + \gamma_{27})(a_{s,26} +  a_{s,27}) +   (\gamma_{28} + \gamma_{30})(a_{s,28} + a_{s,30})\\
&\quad  +  (\gamma_{35} + \gamma_{37})(a_{s,35}+  a_{s,37})  +  (\gamma_{29} + \gamma_{36} + \gamma_{40})(a_{s,36}+  a_{s,40})\\
&\quad +  (\gamma_{39} + \gamma_{41})(a_{s,39} + a_{s,41}) +  (\gamma_{42} + \gamma_{43} + \gamma_{44})(a_{s,43} + a_{s,44}) \\
&\quad  +  (\gamma_{29} + \gamma_{42})(a_{s,38} + a_{s,45}) \equiv 0,\\
\rho_3(f) + f &\equiv (\gamma_{25} + \gamma_{26})(a_{s,25} +  a_{s,26}) +   (\gamma_{28} + \gamma_{29})(a_{s,28} + a_{s,29})\\
&\quad +  (\gamma_{35} + \gamma_{36})(a_{s,35} +  a_{s,36}) +  (\gamma_{30} + \gamma_{37} + \gamma_{39})(a_{s,37}  + a_{s,39})\\
&\quad  +  (\gamma_{30} + \gamma_{38} + \gamma_{42})(a_{s,38} + a_{s,42}) +  (\gamma_{40} + \gamma_{41})(a_{s,40} + a_{s,41})\\
&\quad +  (\gamma_{30} + \gamma_{44} + \gamma_{45})(a_{s,44} +  a_{s,45}) \equiv 0.
\end{align*}
The above equalities imply $\gamma_j =0$ for $j = 35, 36, 37, 38, 44$ and $\gamma_j = \gamma_{25}$ for $j \ne 35, 36, 37, 38, 44$. The lemma is proved.
\end{proof}
\begin{proof}[Proof of Proposition \ref{mdd411}] Let $f \in P_4$ such that $[f] \in (QP_4)_{2^{s+1}-3}^{GL_4}$. Since $\Sigma_4 \subset GL_4$, we have $[f] \in (QP_4)_{2^{s+1}-3}^{\Sigma_4}$. Then, $f \equiv \sum_{j=1}^4\gamma_jp_{4,s,j}$ with $\gamma_j \in \mathbb F_2$. By a direct computation, we get
\begin{align*}
\rho_4(f) + f &\equiv  (\gamma_{1} + \gamma_{4})a_{s,3} + \gamma_{1}a_{s,9} + (\gamma_{2} + \gamma_{3})a_{s,15}\\
&\quad + \gamma_{2}a_{s,21} + \mbox{ other terms } \equiv 0.
\end{align*}
The last equality implies $\gamma_j = 0$ for $j = 1,2,3,4$. The proposition follows.
\end{proof}
From Theorem \ref{md51}, we see that $\text{Ext}_{\mathcal A}^{4, 2^{s+1}+1}(\mathbb F_2, \mathbb F_2) = 0$ for any $s \ne 5$ and $\text{Ext}_{\mathcal A}^{4, 65}(\mathbb F_2, \mathbb F_2) = \langle D_3(0)\rangle$. Hence, Theorem \ref{dlm} holds for $n = 2^{s+1}-3$.

\subsection{The case $n = 2^{s+1}-2$}\

\medskip
Since Kameko's homomorphism in the degree $2^{s+1}-2$,
$$(\widetilde{Sq}^0_*)_{2^{s+1}-2} : (QP_4)_{2^{s+1}-2} \to (QP_4)_{2^{s} - 3}$$ 
is an epimorphism of $GL_4$-modules, using Proposition \ref{mdd411}, we have
$$(QP_4)_{2^{s+1}-2}^{GL_4} \subset \big(\text{Ker}(\widetilde{Sq}^0_* )_{2^{s+1}-2}\big)^{GL_4}.$$

From \cite{su1,su5}, we have the following.
\begin{props}[see \cite{su1,su5}]\label{mdd42} Let $s$ be a positive integer. Then,
$$\dim\big(\text{\rm Ker}(\widetilde{Sq}^0_* )_{2^{s+1}-2}\big) = \begin{cases} 6, &\text{if } s =1,\\ 20, &\text{if } s =2,\\ 35 &\text{if } s \geqslant 3.\end{cases}
$$
\end{props}

A basis for $\big(\text{\rm Ker}(\widetilde{Sq}^0_* )_{2^{s+1}-2}\big)$ is the set consisting of all the classes represented monomials $b_j= b_{s,j}$ which are determined as follows:

\medskip
For $s \geqslant 1$,

\medskip
 \centerline{\begin{tabular}{lll}
$b_{s,1} = x_3^{2^s-1}x_4^{2^s-1}$ &  $b_{s,2} = x_2^{2^s-1}x_4^{2^s-1}$ &  $b_{s,3} = x_2^{2^s-1}x_3^{2^s-1}$\cr  
$b_{s,4} = x_1^{2^s-1}x_4^{2^s-1}$ &  $b_{s,5} = x_1^{2^s-1}x_3^{2^s-1}$ &  $b_{s,6} = x_1^{2^s-1}x_2^{2^s-1}$\cr  
\end{tabular}}

\medskip
For $s \geqslant 2$,

\medskip
 \centerline{\begin{tabular}{lll}
$b_{s,7} = x_2x_3^{2^s-2}x_4^{2^s-1}$ &  $b_{s,8} = x_2x_3^{2^s-1}x_4^{2^s-2}$ &  $b_{s,9} = x_2^{2^s-1}x_3x_4^{2^s-2}$\cr  
$b_{s,10} = x_1x_3^{2^s-2}x_4^{2^s-1}$ &  $b_{s,11} = x_1x_3^{2^s-1}x_4^{2^s-2}$ &  $b_{s,12} = x_1x_2^{2^s-2}x_4^{2^s-1}$\cr  
$b_{s,13} = x_1x_2^{2^s-2}x_3^{2^s-1}$ &  $b_{s,14} = x_1x_2^{2^s-1}x_4^{2^s-2}$ &  $b_{s,15} = x_1x_2^{2^s-1}x_3^{2^s-2}$\cr  
$b_{s,16} = x_1^{2^s-1}x_3x_4^{2^s-2}$ &  $b_{s,17} = x_1^{2^s-1}x_2x_4^{2^s-2}$ &  $b_{s,18} = x_1^{2^s-1}x_2x_3^{2^s-2}$\cr  
\end{tabular}}

\medskip
For $s = 2$, $b_{2,19} = x_1x_2x_3^{2}x_4^{2}$,\  $b_{2,20} = x_1x_2^{2}x_3x_4^{2}$.

\medskip
For $s \geqslant 3$,

\medskip
 \centerline{\begin{tabular}{lll}
$b_{s,19} = x_2^{3}x_3^{2^s-3}x_4^{2^s-2}$ &  $b_{s,20} = x_1^{3}x_3^{2^s-3}x_4^{2^s-2}$ &  $b_{s,21} = x_1^{3}x_2^{2^s-3}x_4^{2^s-2}$\cr  
$b_{s,22} = x_1^{3}x_2^{2^s-3}x_3^{2^s-2}$ &  $b_{s,23} = x_1x_2^{2}x_3^{2^s-4}x_4^{2^s-1}$ &  $b_{s,24} = x_1x_2^{2}x_3^{2^s-1}x_4^{2^s-4}$\cr  
$b_{s,25} = x_1x_2^{2^s-1}x_3^{2}x_4^{2^s-4}$ &  $b_{s,26} = x_1^{2^s-1}x_2x_3^{2}x_4^{2^s-4}$ &  $b_{s,27} = x_1x_2x_3^{2^s-2}x_4^{2^s-2}$\cr  
$b_{s,28} = x_1x_2^{2^s-2}x_3x_4^{2^s-2}$ &  $b_{s,29} = x_1^{3}x_2^{5}x_3^{2^s-6}x_4^{2^s-4}$ &  $b_{s,30} = x_1x_2^{2}x_3^{2^s-3}x_4^{2^s-2}$\cr  
$b_{s,31} = x_1x_2^{3}x_3^{2^s-4}x_4^{2^s-2}$ &  $b_{s,32} = x_1x_2^{3}x_3^{2^s-2}x_4^{2^s-4}$ &  $b_{s,33} = x_1^{3}x_2x_3^{2^s-4}x_4^{2^s-2}$\cr  
$b_{s,34} = x_1^{3}x_2x_3^{2^s-2}x_4^{2^s-4}$ &   &\cr
\end{tabular}}

\medskip
For $s = 3$, $b_{3,35} = x_1^3x_2^3x_3^{4}x_4^{4}$, and 
for $s \geqslant 4$, $b_{s,35} = x_1^{3}x_2^{2^s-3}x_3^{2}x_4^{2^s-4}$.

We set 
$$p_{4,s} = \begin{cases} x_1x_2x_3^6x_4^6 + x_1^3x_2^3x_3^4x_4^4, &\text{if } s =3,\\  \sum_{j=1}^{35}b_{s,j} &\text{if } s \geqslant 4.\end{cases}$$
By a direct computation using Proposition \ref{mdd42}, one gets the following.
\begin{props}\label{mdd412} Let $s$ be a positive integer. Then,
$$\big(\text{\rm Ker}(\widetilde{Sq}^0_* )_{2^{s+1}-2}\big)^{GL_4} = \begin{cases} 0, &\text{if } s \leqslant 2,\\ \langle [p_{4,s}]\rangle &\text{if } s \geqslant 3.\end{cases}
$$
\end{props}

For simplicity, we will prove this proposition in detail for $s \geqslant 4$. The others can be proved by the similar computations.
We have the following.
\begin{lems}\label{bd2m2} {\rm i)} For any $s\geqslant 4$, there is a direct summand decomposition of the $\Sigma_4$-modules:
$$\big(\text{\rm Ker}(\widetilde{Sq}^0_* )_{2^{s+1}-2} = \Sigma_4(b_{s,1})\oplus \Sigma_4(b_{s,7}) \oplus \Sigma_4(b_{s,19}) \oplus \Sigma_4(b_{s,23}) \oplus \Sigma_4(b_{s,29},b_{s,30}).$$

{\rm ii)} $\Sigma_4(b_{s,1})^{\Sigma_4} = \langle [\bar p_{4,s,1}]\rangle$, with $\bar p_{4,s,1}= \sum_{j=1}^{6}b_{s,j}$.

{\rm iii)} $\Sigma_4(b_{s,7})^{\Sigma_4} = \langle [\bar p_{4,s,2}]\rangle$, with $\bar p_{4,s,2}= \sum_{j=7}^{18}b_{s,j}$.

{\rm iv)} $\Sigma_4(b_{s,19})^{\Sigma_4} =  \langle [\bar p_{4,s,3}]\rangle$, with $\bar p_{4,s,3}= \sum_{j=19}^{22}b_{s,j}$.

{\rm v)} $\Sigma_4(b_{s,23})^{\Sigma_4} =  \langle [\bar p_{4,s,4}]\rangle$, with $\bar p_{4,s,4}= \sum_{j=23}^{26}b_{s,j}$.

{\rm vi)} $\Sigma_4(b_{s,29},a_{s,30})^{\Sigma_4} = \langle [\bar p_{4,s,5}],[\bar p_{4,s,6}]\rangle$, where $$\bar p_{4,s,5}= \sum_{j=27}^{29}b_{s,j},\  \bar p_{4,s,6}=  \sum_{j=30}^{35}b_{s,j}.$$
\end{lems}
\begin{proof} From Proposition \ref{mdd42} we easily obtain Part i). Now, we prove Part vi) in detail. The others can be proved by the similar computations. By a direct computation we see that the set $\{[b_{s,j}]: j = 27 \leqslant j \leqslant 35\}$ is a basis for $\Sigma_4(b_{s,29},b_{s,30})$. Suppose $[f] \in \Sigma_4(b_{s,29},b_{s,30})^{\Sigma_4}$, then 
$ f \equiv  \sum_{j=27}^{35}\gamma_jb_{s,j}$ with $\gamma_j \in \mathbb F_2$. By a direct computation, we obtain
\begin{align*}
\rho_1(f) + f &\equiv (\gamma_{28} + \gamma_{29} + \gamma_{30} + \gamma_{35})b_{s,27} +  (\gamma_{31} + \gamma_{33})(b_{s,31}+ b_{s,33})\\
&\quad +   (\gamma_{32} + \gamma_{34})(b_{s,32}  + b_{s,34}) \equiv 0, \\
\rho_2(f) + f &\equiv (\gamma_{27} + \gamma_{28} + \gamma_{32} + \gamma_{33})(b_{s,27} +  b_{s,28})\\
&\quad +   (\gamma_{30} + \gamma_{31})(b_{s,30}  + b_{s,31}) +  (\gamma_{34} + \gamma_{35})(b_{s,34} + b_{s,35}) \equiv 0,\\
\rho_3(f) + f &\equiv (\gamma_{28} + \gamma_{29} + \gamma_{30} + \gamma_{35})b_{s,27} +  (\gamma_{31} + \gamma_{32})(b_{s,31} + b_{s,32})\\
&\quad +  (\gamma_{33} + \gamma_{34})(b_{s,33} + b_{s,34}) \equiv 0.
\end{align*}
The above equalities imply $\gamma_j =\gamma_{27}$ for $j = 27,28,29$ and $\gamma_j = \gamma_{30}$ for $30 \leqslant j \leqslant 35$. The lemma is proved.
\end{proof}
\begin{rems} For $s = 3$, Parts i) to v) of Lemma \ref{bd2m2} hold. We replace Part vi) with
$\Sigma_4(b_{3,29},a_{3,30})^{\Sigma_4} = \langle [p_{4,3}]\rangle.$
\end{rems}
\begin{proof}[Proof of Proposition \ref{mdd412}] Let $f \in P_4$ such that $[f] \in \text{\rm Ker}(\widetilde{Sq}^0_* )_{2^{s+1}-2}^{GL_4}$. Then, $[f] \in \text{\rm Ker}(\widetilde{Sq}^0_* )_{2^{s+1}-2}^{\Sigma_4}$. Hence, $f \equiv \sum_{j=1}^6\gamma_j\bar p_{4,s,j}$ with $\gamma_j \in \mathbb F_2$. By a direct computation, we  have
\begin{align*}
\rho_4(f) + f &\equiv (\gamma_{1} + \gamma_{2})(b_{s,2} +  b_{s,3}) +   (\gamma_{2} + \gamma_{4})(b_{s,7} +  b_{s,8}) +  (\gamma_{2} + \gamma_{5})b_{s,9}\\
&\quad +  (\gamma_{2} + \gamma_{3})(b_{s,14} +  b_{s,15}) +  (\gamma_{3} + \gamma_{6})b_{s,19} +  (\gamma_{4} + \gamma_{6})b_{s,25}\\
&\quad +  (\gamma_{2} + \gamma_{3} + \gamma_{4} + \gamma_{5})b_{s,27} +  (\gamma_{5} + \gamma_{6})(b_{s,31} + b_{s,32}) \equiv 0.
\end{align*}
The last equality implies $\gamma_j = \gamma_1$ for $1 \leqslant j \leqslant 6$. The proposition follows.
\end{proof}
From Theorem \ref{md51}, we have 
$$\text{Ext}_{\mathcal A}^{4, 2^{s+1}+2}(\mathbb F_2, \mathbb F_2) = \begin{cases} 0, &\text{if } s \leqslant 2,\\ 
\langle d_0\rangle &\text{if } s = 3,\\  
\langle h_0^2h_6^2, D_3(1)\rangle &\text{if }  s = 6,\\
\langle h_0^2h_s^2\rangle &\text{if } s \geqslant 4,\, s \ne 6.\\
\end{cases}
$$
Denote by $q_{4,s} \in P((P_4^*)_{2^{s+1}-2})$ the dual of $p_{4,s} \in \big(\text{\rm Ker}(\widetilde{Sq}^0_* )_{2^{s+1}-2}\big)^{GL_4}$. Then, we have
$$\mathbb F_2{\otimes}_{GL_4} P((P_4)^*_{2^{s+1} -2}) = \begin{cases} 0, &\text{if } s \leqslant 2,\\ \langle [q_{4,s}]\rangle, &\text{if } s \geqslant 3.\end{cases}$$

From H\`a \cite{ha} and Singer \cite{si1}, we see that $d_0,\, h_0^2h_s^2 \in \mbox{Im}(Tr_4)$, hence we get
$$Tr_4([q_{4,s}]) = [\phi_4(q_{4,s})] = \begin{cases} d_0, &\text{if } s = 3,\\ h_0^2h_{s}^2, &\text{if } s > 3.\end{cases}$$
Theorem \ref{dlm} holds for $n = 2^{s+1}-2$.

\begin{rems} a) It is well-known that the space $\mbox{Im}(Tr_4)$ had been explicitly determined from the works Bruner, H\`a and H\uhorn ng~\cite{br}, Ch\ohorn n and H\`a~\cite{cha2},  H\`a ~\cite{ha}, H\uhorn ng and Qu\`ynh~ \cite{hq}, Nam~ \cite{na2} and Singer \cite{si1}. Hence, the proof that a certain element is in $\mbox{Im}(Tr_4)$ is unnecessary (see \cite{su5}). To illustrate the fact that $d_0 \in \mbox{Im}(Tr_4)$, we present the computations of H\`a \cite{ha} for this result.
	 
In \cite[Page 102]{ha}, H\`a showed that the element $q_{4,3} \in P((P_4^*)_{14})$ can be determined by
\begin{align*}q_{4,3} &= a_1^{(1)}a_2^{(1)}a_3^{(6)}a_4^{(6)} + a_1^{(1)}a_2^{(2)}a_3^{(5)}a_4^{(6)} + a_1^{(1)}a_2^{(3)}a_3^{(4)}a_4^{(6)} + a_1^{(1)}a_2^{(4)}a_3^{(3)}a_4^{(6)}\\
&\quad + a_1^{(1)}a_2^{(5)}a_3^{(2)}a_4^{(6)} + a_1^{(1)}a_2^{(6)}a_3^{(1)}a_4^{(6)} + a_1^{(2)}a_2^{(1)}a_3^{(6)}a_4^{(5)} + a_1^{(2)}a_2^{(2)}a_3^{(5)}a_4^{(5)}\\
&\quad + a_1^{(2)}a_2^{(3)}a_3^{(4)}a_4^{(5)} + a_1^{(2)}a_2^{(4)}a_3^{(3)}a_4^{(5)} + a_1^{(2)}a_2^{(5)}a_3^{(2)}a_4^{(5)} + a_1^{(2)}a_2^{(6)}a_3^{(1)}a_4^{(5)}\\
&\quad + a_1^{(3)}a_2^{(1)}a_3^{(5)}a_4^{(5)} + a_1^{(3)}a_2^{(2)}a_3^{(6)}a_4^{(3)} + a_1^{(3)}a_2^{(3)}a_3^{(2)}a_4^{(6)} + a_1^{(3)}a_2^{(4)}a_3^{(1)}a_4^{(6)}\\
&\quad + a_1^{(3)}a_2^{(4)}a_3^{(2)}a_4^{(5)} + a_1^{(3)}a_2^{(4)}a_3^{(4)}a_4^{(3)} + a_1^{(3)}a_2^{(6)}a_3^{(2)}a_4^{(3)} + a_1^{(4)}a_2^{(1)}a_3^{(6)}a_4^{(3)}\\
&\quad + a_1^{(4)}a_2^{(2)}a_3^{(5)}a_4^{(3)} + a_1^{(4)}a_2^{(3)}a_3^{(4)}a_4^{(3)} + a_1^{(4)}a_2^{(4)}a_3^{(3)}a_4^{(3)} + a_1^{(4)}a_2^{(5)}a_3^{(2)}a_4^{(3)}\\
&\quad + a_1^{(4)}a_2^{(6)}a_3^{(1)}a_4^{(3)} + a_1^{(5)}a_2^{(1)}a_3^{(3)}a_4^{(5)} + a_1^{(5)}a_2^{(2)}a_3^{(1)}a_4^{(6)} + a_1^{(5)}a_2^{(2)}a_3^{(2)}a_4^{(5)}\\
&\quad + a_1^{(5)}a_2^{(2)}a_3^{(4)}a_4^{(3)} + a_1^{(5)}a_2^{(3)}a_3^{(1)}a_4^{(5)} + a_1^{(5)}a_2^{(3)}a_3^{(3)}a_4^{(3)} + a_1^{(5)}a_2^{(5)}a_3^{(1)}a_4^{(3)}\\
&\quad + a_1^{(6)}a_2^{(1)}a_3^{(1)}a_4^{(6)} + a_1^{(6)}a_2^{(1)}a_3^{(2)}a_4^{(5)} + a_1^{(6)}a_2^{(1)}a_3^{(4)}a_4^{(3)} + a_1^{(6)}a_2^{(2)}a_3^{(3)}a_4^{(3)}, 
\end{align*}
By a direct computation we can easily verify that $Sq^1(q_{4,3}) = 0$, $Sq^2(q_{4,3}) = 0$, $Sq^4(q_{4,3}) = 0$. Hence, $q_{4,3} \in P((P_4^*)_{14})$.

Ch\ohorn n showed in his PhD thesis that
$\phi_4(q_{4,3}) = \bar d_0 + \delta(\lambda_3^2\lambda_9 + \lambda_3\lambda_9\lambda_3)$. Hence, one gets $Tr_4([q_{4,3}]) = [\phi_4(q_{4,3})] = [\bar d_0] = d_0$.

So, it is possible the algorithm in \cite{p725} is flawed.

b) In \cite{hq}, H\uhorn ng and Qu\`ynh stated that $p_0 \in \mbox{Im}(Tr_4)$ but did not provide the detailed proof. However, this result is explicitly proved in Ch\ohorn n and H\`a~\cite{cha2}. Hence, the computations in \cite{p725} for $p_0$ may be new but they are unnecessary for studying Singer's conjecture.

c) In \cite{suw}, we have given a negative answer for Singer's conjecture for the algebraic transfer. Hence, Singer's algebraic transfer cannot be a tool for studying the cohomology of Steenrod algebra. Therefore, the study of Singer's algebraic transfer is no longer necessary.
\end{rems}
\subsection{The case $n = 2^{s+1}-1$}\

\medskip
First, we recall the following.
\begin{props}[see \cite{su1,su5}]\label{mdd43}Let $n = 2^{s+1}-1$ with $s$ a positive integer. Then, the dimension of the $\mathbb F_2$-vector space $(QP_4)_{n}$ is determined by the following table:

\medskip
\centerline{\begin{tabular}{c|ccccc}
$n = 2^s-1$&$s=1$ & $s=2$ & $s=3$&$s=4$  & $s\geqslant 5$\cr
\hline
\ $\dim(QP_4)_n$ & $14$ & $35$ & $75$ &$89$& $85$ \cr
\end{tabular}}
\end{props}
A basis of $(QP_4)_n$ has been given in \cite{su5}. For $s \geqslant k-2$, we set
$$\eta_{k,s} = \sum_{m=1}^{k-1}\sum_{1 \leqslant i_1< \ldots <i_m \leqslant k}x_{i_1}x_{i_2}^2\ldots x_{i_{m-1}}^{2^{m-2}}x_{i_m}^{2^{s+1}-2^{m-1}}\in (P_k)_{2^{s+1}-1}.$$
For $k = 4$, we denote 
$$\bar p_{4,s} = \begin{cases} \eta_{4,s} + x_1x_2^2x_3^2x_4^2, &\text{if } s =2,\\ 
\eta_{4,s} +  x_1x_2^2x_3^4x_4^{2^{s+1}-8},  &\text{if } s \geqslant 3.\end{cases}
$$
By a computation similar to the one in Proposition \ref{mdd42}, one gets the following.
\begin{props}\label{mdd431} Let $s$ be a positive integer. Then,
$$(QP_4)_{2^{s+1}-1}^{GL_4} = \begin{cases} 0, &\text{if } s =1,\\ \langle [\bar p_{4,s}]\rangle ,  &\text{if } s \geqslant 2.\end{cases}
$$
\end{props}

From Theorem \ref{md51}, we have 
$$\text{Ext}_{\mathcal A}^{4, 2^{s+1}+3}(\mathbb F_2, \mathbb F_2) = \begin{cases} 0, &\text{if } s = 1,\\   \langle h_0^3h_{s+1}\rangle &\text{if } s \geqslant 2.\end{cases}
$$
Denote $\bar q_{4,s} = a_1^{(0)}a_2^{(0)}a_2^{(0)}a_4^{(2^{s+1}-1)} \in P((P_4^*)_{2^{s+1}-1})$, for $s \geqslant 2$. It is easy to see that 
$\langle [\bar p_{4,s}],[\bar q_{4,s}] \rangle =1$
 Hence, we obtain
$$\mathbb F_2{\otimes}_{GL_4} P((P_4)^*_{2^{s+1} -1}) = \begin{cases} 0, &\text{if } s =1,\\ \langle [\bar q_{4,s}]\rangle, &\text{if } s \geqslant 2.\end{cases}$$
By a simple computation, we have $\phi_4(\bar q_{4,s}) = \lambda_0^3\lambda_{2^{s+1}-1}$. Hence, using Theorem \ref{md52}, one gets  
$$Tr_4([\bar q_{4,s}]) = [\phi_4(\bar q_{4,s})] = [\lambda_0^3\lambda_{2^{s+1}-1}] = h_0^3h_{s+1}.$$
Theorem \ref{dlm} is completely proved.

\section*{Acknowledgment} 

This article was written when the author was visiting the Vietnam Institute for Advanced Study in Mathematics (VIASM) from August to November 2017. He would like to thank the VIASM for supporting the visit, convenient working condition and for kind hospitality.

The author would like to express his warmest thanks to the referee for carefully reading the manuscript and giving many criticisms and suggestions, which have led to an improvement of the article's exposition.

\bibliographystyle{amsplain}

\end{document}